\documentclass[11pt,oneside,a4paper]{article}

\usepackage[english]{babel}	
\usepackage[T1]{fontenc}		
\usepackage[utf8]{inputenc}	
\usepackage{lmodern}	
\usepackage{amsmath}
\usepackage{amsthm}	
\usepackage{amssymb}
\usepackage{mathrsfs}
\usepackage{paralist}	
\usepackage{geometry}			
\usepackage[all]{xy}
\usepackage{xcolor}
\usepackage[colorlinks = true,
            linktocpage = true,
            linkcolor = blue,
            urlcolor  = blue,
            citecolor = blue,
            anchorcolor = black]{hyperref}

\theoremstyle{plain}
\newtheorem{thm}{Theorem}
\newtheorem{prop}[thm]{Proposition}
\newtheorem{lem}[thm]{Lemma}
\newtheorem{cor}[thm]{Corollary}

\newcommand{\N}{\mathbb{N}}
\newcommand{\Z}{\mathbb{Z}}
\newcommand{\R}{\mathbb{R}}
\newcommand{\C}{\mathbb{C}}
\newcommand{\F}{\mathscr{F}}
\newcommand{\M}{\mathscr{M}}
\newcommand{\G}{\mathscr{G}}
\newcommand{\K}{\mathscr{K}}
\newcommand{\del}{\partial}

\DeclareMathOperator{\Vol}{Vol}
\DeclareMathOperator{\im}{im}
\DeclareMathOperator{\PD}{PD}
\DeclareMathOperator{\re}{Re}

\date{}
\author{\textsc{Thibault Langlais}\footnote{Mathematical Institute, University of Oxford, Oxford OX2 6GG, United Kingdom. \newline E-mail address: langlais@maths.ox.ac.uk. ORCID iD: \href{https://orcid.org/0000-0002-6434-2988}{0000-0002-6434-2988}.} }
\title{\bfseries{On the incompleteness of $G_2$-moduli spaces  \\ along degenerating families of $G_2$-manifolds}}

\begin{document}

\maketitle

\begin{abstract}
    \noindent{}We derive a formula for the energy of a path in the moduli space of a compact $G_2$-manifold with vanishing first Betti number for the volume-normalised $L^2$-metric. This allows us to give simple sufficient conditions for a path of torsion-free $G_2$-structures to have finite energy and length. We deduce that the compact $G_2$-manifolds produced by the generalised Kummer construction have incomplete moduli spaces. Under some assumptions, we also state a necessary condition for the limit of a path of torsion-free $G_2$-structures to be at infinite distance in the moduli space.
\end{abstract}

    \section{Introduction and motivation}      \label{section:intro}

The aim of this short paper is to provide a simple proof that $G_2$-moduli spaces are not generally complete. More precisely, we prove that certain singular limits of degenerating families of $G_2$-manifolds occur at finite distance in the moduli spaces. To the author's knowledge, this was previously unknown. Before explaining our strategy, we recall some background on $G_2$-geometry and give some motivation.

Let $M$ be an oriented $7$-manifold. A \emph{$G_2$-structure} on $M$ corresponds to the data of a $3$-form $\varphi$ satisfying the following positivity condition:
\begin{equation*}
    (u \lrcorner \varphi_x) \wedge (u \lrcorner \varphi_x) \wedge \varphi_x > 0, ~~~ \forall x \in M, ~ \forall u \in T_x M \backslash \{ 0 \},
\end{equation*}
where $\lrcorner$ denotes the interior product and positivity is relative to the choice of orientation. A $G_2$-structure $\varphi$ naturally induces a Riemannian metric $g_\varphi$ and a volume form $\Vol_\varphi = \Vol_{g_\varphi}$ on $M$. We denote by $*_\varphi$ the associated Hodge-$*$ operator and by $\Theta(\varphi) = *_\varphi \varphi$ the dual $4$-form of $\varphi$. A $G_2$-structure $\varphi$ is said to be \emph{torsion-free} if it is parallel with respect to the Levi--Civita connection of $g_\varphi$. An equivalent condition is to require that $\varphi$ be closed and co-closed \cite{fernandez1982riemannian}. If $\varphi$ is torsion-free, then the holonomy group of the metric $g_\varphi$ is conjugate to a subgroup of $G_2 \subset SO(7)$ in $GL(7)$, which implies that $g_\varphi$ is Ricci-flat \cite{bonan1966sur}. A manifold endowed with a torsion-free $G_2$-structure is called a \emph{$G_2$-manifold}.

One motivation for studying $G_2$-manifolds is that $G_2$, along with $Spin(7) \subset SO(8)$, is one of the two exceptional cases in Berger's list of Riemannian holonomy groups \cite{berger1955groupes} (the original list also contained $Spin(9) \subset SO(16)$, but it was later ruled out in \cite{alekseevsky1968riemannian}). The existence of metrics with full holonomy $G_2$ was first proved by Bryant \cite{bryant1987metrics} using local methods, shortly before the first complete noncompact examples were exhibited by Bryant--Salamon \cite{bryant1989construction}. The first compact examples are due to Joyce \cite{joyce1996compacti,joyce1996compactii}, with a method modelled on the Kummer construction of K3 surfaces \cite{page1978physical,topiwala1987new,lebrun1994kummer}, hence called the \emph{generalised Kummer construction}. Several other constructions were subsequently developed \cite{joyce2000compact,kovalev2003twisted,karigiannis2009desingularization,corti2013asymptotically,corti2015g,joyce2021new,nordstrom2023extra}. At present, all of the known constructions use different geometric methods to produce a compact manifold endowed with a closed $G_2$-structure with small torsion, before applying the general analytical machinery developed by Joyce \cite[Ch. 11]{joyce2000compact} to deform it to a nearby torsion-free $G_2$-structure within the same cohomology class.

Let us now consider a compact oriented $7$-manifold $M$ admitting torsion-free $G_2$-structures. The \emph{moduli space $\M$ of torsion-free $G_2$-structures} on $M$ is defined as the quotient of the space of torsion-free $G_2$-structures by the action of the group of diffeomorphisms of $M$ acting trivially on $H^3(M)$. It has the structure of a smooth manifold of dimension $b^3(M)$, and for any torsion-free $G_2$-structure $\varphi$ the tangent space $T_\varphi \M$ can be identified with the space of harmonic $3$-forms with respect to the metric $g_\varphi$ \cite{joyce1996compacti}. Moreover, there is a local diffeomorphism $\M \rightarrow H^3(M)$ which maps $\varphi \in \M$ to the cohomology class $[\varphi] \in H^3(M)$. This endows $\M$ with the structure of an affine manifold, and in particular $\M$ has a flat connection $D$ obtained by pulling back the natural flat connection of $H^3(M)$. 

When $b^1(M) = 0$, the moduli space $\M$ can be equipped with a Hessian structure. Let $\Vol(\varphi)$ be the volume of $M$ for the metric induced by a $G_2$-structure $\varphi$. As the volume functional is invariant under diffeomorphisms, it descends to a smooth function on the moduli space. We define a smooth function $\F : \M \rightarrow \R$ by
\begin{equation*}
    \F(\varphi) = - 3 \log \Vol(\varphi) .
\end{equation*}
This function owes its interest to the properties of the first and second derivatives \cite{grigorian2009local,karigiannis2009hodge}. The differential $d\F = D\F$ at $\varphi \in \M$ is given by
\begin{equation}    \label{eq:fa}
    d_\varphi \F(\eta) = - \frac{1}{\Vol(\varphi)} \int \eta \wedge \Theta(\varphi) = - \frac{\langle [\eta] \cup [\Theta(\varphi)], [M] \rangle}{\Vol(\varphi)}, ~~~ \forall \eta \in T_{\varphi} \M .
\end{equation}
When $b^1(M) = 0$ the Hessian $Dd\F = D^2\F$ with respect to the flat connection $D$ reads
\begin{equation}    \label{eq:fab}
    D^2_\varphi\F(\eta,\eta^\prime) = \frac{1}{\Vol(\varphi)} \int \eta \wedge *_\varphi \eta^\prime, ~~~ \forall \eta,\eta^\prime \in T_{\varphi} \M .
\end{equation}
Hence the Hessian of $\F$ is positive-definite, and therefore defines a Riemannian metric $\G$ on $\M$, which is just the volume-normalised $L^2$-metric. Thus $(\M,D,\G)$ has the structure of a Hessian manifold, admitting $\F$ as a global potential. Note that it is crucial that $b^1(M) = 0$, for in general the Hessian of $\F$ has signature $(b^3(M) - b^1(M),b^1(M))$. In the present article we are interested in manifolds with full holonomy, which have $b^1(M) = 0$ since their fundamental group is finite \cite{joyce2000compact}. Let us also point out that $(\M,\G)$ is isometric to $\R \times (\M_0,\G_0)$, where $\M_0$ is the moduli space of unit volume torsion-free $G_2$-structures and $\G_0$ the $L^2$-metric. In particular any limit of $G_2$-manifolds with volume diverging to zero or infinity is at infinite distance, and $(\M,\G)$ is complete if and only if $(\M_0,\G_0)$ is.

The interest in the metric $\G$ is notably motivated by the importance of $G_2$-manifolds for M-theory in physics, where $G_2$-compactifications play the same role as Calabi--Yau compactifications in string theory. In this context, $\G$ arises as a kinetic term in the low-energy effective action, in the same way as the natural metrics carried by the K\"{a}hler cone or the moduli space of complex structures of a Calabi--Yau manifold arise from string theory. As part of the swampland programme formulated by Vafa \cite{vafa2005string}, it has been conjectured that all of these metrics should share some common properties. An important component of this programme, related to the so-called swampland distance conjecture \cite{ooguri2007geometry}, would be to characterise which limits are at finite or infinite distance in the moduli spaces. 

A well-understood case is that of the K\"{a}hler cone of a compact K\"{a}hler manifold, where the natural metric also admits a global Hessian potential. The K\"{a}hler cone can be described in terms of the intersection form and the classes of analytic cycles \cite{demailly2004numerical}, and there is a simple necessary and sufficient condition for a cohomology class at the boundary of the cone to be a finite-distance limit \cite{magnusson2012metriques}. Examples where the K\"{a}hler cone is incomplete include Kummer K3 surfaces, where a sequence of hyperk\"{a}hler metrics degenerating to $T^4/\Z_2$ occurs at finite distance. In the case of moduli spaces of polarised Calabi--Yau manifolds, the Weil--Petersson metric can be studied using Hodge-theoretic methods, and there are known examples where the moduli spaces are incomplete as well as some characterisations of finite-distance degenerations \cite{wang1997incompleteness}. These results rely on techniques of complex algebraic geometry, together with the link to Riemannian geometry provided by Yau's solution of the Calabi conjecture \cite{yau1978ricci}. In comparison, $G_2$-manifolds are only amenable to differential-geometric methods and very little is known about the global properties of $(\M,\G)$, except that in some cases $\M$ (and even the quotient of the space of torsion-free $G_2$-structures by the full group of diffeomorphisms) is disconnected \cite{crowley2015analytic}. In this article we will show that the generalised Kummer $G_2$-manifolds have incomplete moduli spaces, as can be expected by analogy with the case of Kummer K3 surfaces.

To prove that the moduli space of a certain $G_2$-manifold is incomplete, the idea is to find a path of torsion-free $G_2$-structures degenerating towards a singular limit and prove that it has finite length in the moduli space. The natural paths to consider are those constructed by gluing-perturbation methods, which are typically indexed by a parameter representing the size of the gluing region. To compute the length of a such a path, we a priori need to differentiate the family of torsion-free $G_2$-structures with respect to the gluing parameter to deduce the velocity vector along the corresponding path in the moduli space and estimate its $L^2$-norm. There is ongoing work by J. Li \cite{li2023resolution} using this approach, but there are many analytical difficulties related to the fact that the torsion-free $G_2$-structures obtained after perturbation are only implicitly defined, making this method hard to implement in detail.

Here we adopt a much simpler approach, in which the analytical difficulties disappear. To circumvent the analysis, the idea is to consider not the length but the energy of a path and make use of the special properties of the metric $\G$. In the next section, we derive an expression for the energy of a curve that involves derivatives only of the cohomology class of the $3$-form and no derivatives of the cohomology class of its dual $4$-form. Interpreting this expression in geometrical terms allows us to give simple sufficient conditions for a path of torsion-free $G_2$-structures to have finite energy and length. In Sect. \ref{section:gluing} we use these conditions to prove that, for the generalised Kummer $G_2$-manifolds constructed in \cite{joyce1996compacti,joyce1996compactii}, the degeneration to a flat orbifold $T^7/\Gamma$ occurs at finite distance in the moduli space. In Sect. \ref{section:infinite}, we consider the case of a path of torsion-free $G_2$-structures whose cohomology classes form a line segment in $H^3(M)$, and prove that if the limit of such a path is at infinite distance then either the volume of $M$ is shrinking to zero or there is a homology class in degree $4$ whose volume is going to infinity. We also discuss how this result compares to the case of K\"{a}hler cones and possible generalisations.

    \section{Length and energy of curves in the moduli space}      \label{section:energy}

Let $M$ be a compact oriented $7$-manifold with $b^1(M)=0$ admitting torsion-free $G_2$-structures. We aim to compute the energy of a path in the moduli space $\M$ for the metric $\G$. We begin with a few remarks about notations and the regularity of paths. 

We will say that a family of torsion-free $G_2$-structures $\{ \varphi_t\}_{t \in (0,T]}$ on $M$ \emph{induces a path of class $C^k$ in $\M$} ($k \in \N \cup \{ \infty \}$) if the path defined by the class of $\varphi_t$ modulo the group of diffeomorphisms acting trivially on $H^3(M)$ is of regularity $C^k$ in $\M$. We emphasize that we do not require the $3$-forms $\varphi_t$ to be of class $C^k$ jointly in the variable $t$ and in local coordinates on $M$, since at no point will we need to consider partial derivatives of $\varphi_t$ with respect to the variable $t$. In practice, if $\varphi_t$ is a continuous family of $G_2$-structures on $M$ and the cohomology class $[\varphi_t]$ defines a path of class $C^k$ in $H^3(M)$, then $\{\varphi_t\}_{t \in (0,T]}$ induces a path of class $C^k$ in $\M$, since the moduli space is locally diffeomorphic to $H^3(M)$. Although the energy is defined for a path in $\M$ which is merely $C^1$, we will need to consider paths of class at least $C^2$ so as to differentiate and integrate by parts some expressions.

If $\{\varphi_t\}_{t \in (0,T]}$ induces a path of class $C^2$ in $\M$, we denote by $\dot \varphi_t \in T_{\varphi_t}\M$ the velocity vector along the induced path in $\M$ and by $\ddot \varphi_t = \frac{D}{dt} \dot \varphi_t \in T_{\varphi_t} \M$ the covariant derivative of $\dot \varphi_t$ along the path for the flat connection $D$. In particular $[\dot \varphi_t] = \frac{d[\varphi_t]}{dt}$ and $[\ddot \varphi_t] = \frac{d^2[\varphi_t]}{dt^2}$.

\begin{prop}    \label{prop:energy}
    Let $\{\varphi_t\}_{t \in (0,T]}$ be a family of torsion-free $G_2$-structures on $M$, inducing a path of class $C^2$ in $\M$. Then for any $\tau \in (0,T]$ the energy of the path $\{\varphi_t \}_{t \in [\tau,T]}$ reads:
    \begin{multline*}
        E_{\tau}^T(\varphi_t) = \frac{1}{\Vol(\varphi_\tau)} \left\langle \left. \frac{d [\varphi_t]}{dt} \right|_{t=\tau} \cup [\Theta(\varphi_\tau)], [M] \right\rangle - \frac{1}{\Vol(\varphi_T)} \left\langle \left. \frac{d [\varphi_t]}{dt} \right|_{t=T} \cup [\Theta(\varphi_T)], [M] \right\rangle \\ + \int_{\tau}^T \frac{1}{\Vol(\varphi_t)}  \left\langle \frac{d^2 [\varphi_t]}{dt^2} \cup [\Theta(\varphi_t)], [M] \right\rangle dt .
    \end{multline*}
\end{prop}

\begin{proof}
    Let us define the function $h : (0,T] \rightarrow \R$ by
    \begin{equation*}
        h(t) = \frac{1}{\Vol(\varphi_t)} \left\langle \frac{d [\varphi_t]}{dt} \cup [\Theta(\varphi_t)], [M] \right\rangle \cdot
    \end{equation*}
    By \eqref{eq:fa} we have
    \begin{equation*}
        h(t) = - d_{\varphi_t} \F (\dot \varphi_t).
    \end{equation*}
    As $\F$ is smooth and $\{\varphi_t\}_{t \in (0,T]}$ induces a path of class $C^2$ in $\M$, it follows that $h$ is $C^1$. Thus we can differentiate the above equality with respect to $t$, which yields
    \begin{equation*}
        h^\prime(t) = - D^2_{\varphi_t} \F (\dot \varphi_t, \dot \varphi_t) - d_{\varphi_t} \F(\ddot \varphi_t) = - \G_{\varphi_t}(\dot \varphi_t, \dot \varphi_t) - d_{\varphi_t} \F(\ddot \varphi_t) .
    \end{equation*}
    Using \eqref{eq:fa} again we see that
    \begin{equation*}
        d_{\varphi_t} \F(\ddot \varphi_t) = - \frac{1}{\Vol(\varphi_t)} \left \langle \frac{d^2 [\varphi_t]}{dt^2} \cup [\Theta(\varphi_t)] , [M] \right \rangle
    \end{equation*}
    and thus we obtain the identity
    \begin{equation*}
        \G_{\varphi_t}(\dot \varphi_t, \dot \varphi_t) = - h^\prime(t) + \frac{1}{\Vol(\varphi_t)} \left \langle \frac{d^2 [\varphi_t]}{dt^2} \cup [\Theta(\varphi_t)] , [M] \right \rangle .
    \end{equation*}
    Integrating this identity between $\tau$ and $T$ yields the desired result.
\end{proof}

By the Cauchy--Schwarz inequality, the length and energy of the curve $\{\varphi_t\}_{t \in (0,T]}$ satisfy the inequality $L_0^T(\varphi_t)^2 \leq T E_0^T(\varphi_t)$ and therefore we immediately deduce:

\begin{cor}     \label{cor:incompbounds}
    With the same assumptions as in the previous proposition, assume that there exist a constant $C > 0$ and a nonnegative integrable function $A : (0,T] \rightarrow \R$ such that for all $t \in (0,T]$ we have
    \begin{equation*}
        \left| \left \langle \frac{d [\varphi_t]}{dt} \cup [\Theta(\varphi_t)] , [M] \right \rangle \right| \leq C \Vol(\varphi_t) ~~ \text{and} ~~ \left| \left \langle \frac{d^2 [\varphi_t]}{dt^2} \cup [\Theta(\varphi_t)] , [M] \right \rangle \right| \leq A(t) \Vol(\varphi_t) .
    \end{equation*}
    Then the energy and the length of $\{\varphi_t\}_{t \in (0,T]}$ are finite.
\end{cor}

We finish this part with a few remarks about (co)homology groups. If $M$ a smooth manifold, we denote by  $H^p(M)$ the de Rham cohomology groups, which are isomorphic to the singular cohomology groups with real coefficients. The singular homology groups with real coefficients are denoted by $H_p(M)$ and $\langle \cdot, \cdot \rangle$ denotes the natural pairing between $H^p(M)$ and $H_p(M)$. Any homology class $[C] \in H_p(M)$ can be represented by a smooth singular $p$-cycle $C = \sum a_i \sigma_i$, where $\sigma : \Delta_p \rightarrow M$ are smooth $p$-simplices and $\Delta_p$ denotes the standard oriented $p$-simplex. If $\eta$ is a closed $p$-form on $M$, we have
\begin{equation*}
    \langle [\eta], [C] \rangle = \sum a_i \int_{\Delta_p} \sigma_i^* \eta . 
\end{equation*}
In particular the right-hand side does not depend on a particular choice of representative for $[C]$, and for this reason we will denote (with a slight abuse of notation) $\int_{[C]} \eta = \langle [\eta], [C] \rangle$. In the remainder of the article, all singular chains are assumed to be smooth.

If $g$ is a Riemannian metric on $M$, the volume of a $p$-chain $C = \sum_i a_i \sigma_i$ is defined by
\begin{equation*}
    \Vol(C,g) = \sum_i |a_i| \int_{\Delta_p} \Vol_{\sigma_i^*g} .
\end{equation*}
If $M$ has dimension $7$ and is endowed with a co-closed $G_2$-structure $\varphi$, the dual $4$-form $\Theta(\varphi)$ is a calibration \cite{joyce2007riemannian}. Hence for any $4$-simplex $\sigma$ we have $\pm \sigma^* \Theta(\varphi) \leq \Vol_{\sigma^* g_\varphi}$, and therefore for any $4$-cycle $D$ we have a bound
\begin{equation}    \label{eq:thetaest}
    \left| \int_{[D]} \Theta(\varphi) \right| \leq \Vol(D,g_\varphi) .
\end{equation}
Note that the left-hand side is topological and independent of the choice of representative of $[D]$, whilst the right-hand side is geometric and depends on the choice of $4$-cycle $D$.
    \section{Gluing constructions and incompleteness}       \label{section:gluing}

In this section, we consider a simple model of gluing construction of compact $G_2$-manifolds. Topologically, it can be described as follows. Let $\overline U$ be a compact oriented $7$-manifold with boundary and denote $U = \overline U \backslash \del \overline U$. We denote by $\Sigma_1,\ldots,\Sigma_m$ the connected components of $\del \overline U$, each of which is a compact oriented $6$-manifold. We may assume that there are disjoint neighbourhoods $W_i$ of $\Sigma_i$, diffeomorphic to $(0,1] \times \Sigma_i$, such that $\overline U_0 = \overline U \backslash \coprod_i W_i$ is a manifold with boundary diffeomorphic to $\overline U$ and a deformation retract of $\overline U$. We denote by $U_0$ its interior, which is diffeomorphic to $U$. For each $i$, we consider a compact oriented $7$-manifold $\overline Y_i$ with boundary $\del \overline Y_i = \Sigma_i$, and let $K_i$ be a compact subset of $Y_i = \overline Y_i \backslash \Sigma_i$ such that $Y_i \backslash K_i \simeq (-1,0) \times \Sigma_i$. We may assume that the orientation induced by $\overline Y_i$ on $\Sigma_i$ is the opposite of the one induced by $\overline U$, and that $K_i$ is a compact manifold with boundary diffeomorphic to $\overline Y_i$ and a deformation retract of $\overline Y_i$. Given this data we construct a compact oriented $7$-manifold $M = (U \coprod_i Y_i) / \sim$ by identifying the $i$-th end of $U$ with the end of $Y_i$; that is, the equivalence relation $\sim$ identifies $(s,x_i) \in (0,1) \times \Sigma_i \subset U$ with $(-s,x_i) \in (-1,0) \times \Sigma_i \subset Y_i$. Thus $U$ and each $Y_i$ can be seen as open subsets of $M$, and moreover $M$ can be decomposed as the disjoint union $\overline U_0 \coprod (\amalg_i (0,1) \times \Sigma_i) \coprod (\amalg_i K_i)$.

The real homology groups of $M$ can be deduced from the Mayer--Vietoris exact sequence of the decomposition $M = U \cup (\coprod_i Y_i)$, which reads:
\begin{equation}    \label{eq:mayviet}
    \cdots \rightarrow \oplus_i H_p(\Sigma_i) \rightarrow H_p(U) \oplus (\oplus_i H_p(Y_i)) \rightarrow H_p(M) \rightarrow \oplus_i H_{p-1}(\Sigma_i) \rightarrow \cdots
\end{equation}
We are mainly interested in $p = 3,4$. As $H_p(U) \simeq H_p(\overline U)$ and $H_p(Y_i) \simeq H_p(\overline Y_i)$, the maps $\oplus_i H_p(\Sigma_i) \rightarrow H_p(U) \oplus (\oplus_i H_p(Y_i))$ in \eqref{eq:mayviet} are determined by the long exact sequences of the pairs $(\overline U,\del \overline U)$ and $(\overline Y_i, \del \overline Y_i)$. A particular role is played by the boundary maps
\begin{equation*}
    \delta_i : H_3(\overline Y_i, \del \overline Y_i) \rightarrow H_2(\del \overline Y_i)
\end{equation*}
coming from the exact sequences of $(\overline Y_i, \del \overline Y_i)$. In the next lemma, we show that if all the boundary maps $\delta_i$ are trivial, then $H_3(M)$ has a basis represented by cycles supported away from the gluing region.

\begin{lem}     \label{lem:support}
    Assume that $\delta_i = 0$ for all $i$. Then there are $3$-cycles $C_1,\ldots,C_n$ supported in $U_0$ and $C_{i,1},\ldots,C_{i,n_i}$ supported in $K_i$ such that:
    \begin{enumerate}[(i)]
        \item The homology classes $[C_k],[C_{ij}]$ form a basis of $H_3(M)$.
        \item If $[D_k],[D_{ij}] \in H_4(M)$ is the dual basis for the intersection product, then the classes $[D_{ij}]$ can be represented by cycles supported in $K_i$. 
    \end{enumerate}
\end{lem}

\begin{proof}
    Since $\delta_i = 0$, the maps $H_3(\overline Y_i) \rightarrow H_3(\overline Y_i,\del \overline Y_i)$ are surjective. Thus we deduce that $H_3(Y_i) \simeq H_3(\overline Y_i) \simeq \im(H_3(\del \overline Y_i) \rightarrow H_3(\overline Y_i)) \oplus E_i$, where $E_i \subset H_3(Y_i)$ is isomorphic to $H_3(\overline Y_i,\del \overline Y_i)$. As the maps $H_2(\del \overline Y_i) \rightarrow H_2(\overline Y_i)$ are injective, so is the map $\oplus_i H_2(\Sigma_i) \rightarrow H_2(U) \oplus (\oplus_i H_2(Y_i))$ in the exact sequence \eqref{eq:mayviet}. Thus we obtain an exact sequence
    \begin{equation*}
        \cdots \rightarrow \oplus_i H_3(\Sigma_i) \rightarrow H_3(U) \oplus (\oplus_i H_3(Y_i)) \rightarrow H_3(M) \rightarrow 0 .
    \end{equation*}
    Hence there exists a subspace $E \subset H_3(U)$ such that $E \oplus E_1 \oplus \cdots E_m$ is a complement of the image of $ \oplus_i H_3(\Sigma_i)$ in $H_3(U) \oplus (\oplus_i H_3(Y_i))$, and therefore $H_3(M) \simeq E \oplus E_1 \oplus \cdots \oplus E_m$. It follows that there are homology classes $[C_k] \in E \subset H_3(U)$, $k = 1,\ldots,n$, and $[C_{ij}] \in E_i \subset H_3(Y_i)$, $i = 1,\ldots,m$, $j=1,\ldots,n_i$, which form a basis of $H_3(M)$, where $n = \dim E$ and $n_i = \dim E_i$. Moreover, $K_i$ is a deformation retract of $\overline Y_i$ and hence the classes $[C_{ij}]$ can be represented by cycles supported in $K_i$, and similarly the classes $[C_k]$ can be represented by cycles supported in $U_0$ since $H_3(U_0) \simeq H_3(U)$.

    Now we prove that the homology classes $[D_{ij}] \in H_4(M)$ can be represented by cycles supported in $K_i$. As each $\overline Y_i$ is an oriented compact manifold with boundary, there are nondegenerate intersection pairings $H_4(\overline Y_i) \times H_3(\overline Y_i, \del \overline Y_i) \rightarrow \R$. By construction, the basis $[C_{ij}]$ of $E_i \subset H_3(Y_i)$ induces a basis of $H_3(\overline Y_i,\del \overline Y_i)$, and we denote by $[D^\prime_{ij}] \in H_4(\overline Y_i)$ its dual basis for the intersection product of $\overline Y_i$. We can assume that the cycles $D^\prime_{ij}$ are supported in $K_i$ since $H_4(K_i) \simeq H_4(\overline Y_i)$. As the classes $[C_1],\ldots,[C_n] \in H_3(M)$ are represented by cycles supported in $U_0$ and the classes induced by $[D^\prime_{ij}]$ in $H_4(M)$ are represented by cycles supported in $K_i$, the intersection of $[D^\prime_{ij}]$ and $[C_k]$ is trivial in $M$, and the intersection of $[D^\prime_{ij}]$ and $[C_{i^\prime j^\prime}]$ is $1$ if $(i,j) = (i^\prime, j^\prime)$ and $0$ otherwise. Thus $[D_{ij}] = [D^\prime_{ij}] \in H_4(M)$.
\end{proof}

From the $G_2$-perspective, we typically think of $U$ as the smooth locus of a singular $G_2$-manifold, and in particular $U$ comes equipped with a torsion-free $G_2$-structure $\varphi_0$. The noncompact manifolds $Y_i$ are endowed with families $\varphi_{i,t}$ of torsion-free $G_2$-structures with prescribed asymptotic behaviour, which should match the behaviour of $\varphi_0$ near the $i$-th end of $U$. One then uses some interpolation procedure to construct a family of closed $G_2$-structures $\varphi_t$ on $M$, such that outside of the gluing region $\left. \varphi_t \right|_{U_0} = \left. \varphi_0 \right|_{U_0} $ and $\left. \varphi_t \right|_{K_i} = \left. \varphi_{i,t} \right|_{K_i}$. Much of the subtlety of the construction lies in the choice of interpolation in the gluing region, but for our purpose these details are irrelevant. Provided the torsion of $\varphi_t$ is small enough and there is some control on other geometric quantities (notably the injectivity radius and the norm of the curvature tensor), the general result of Joyce \cite[Th. 11.6.1]{joyce2000compact} insures the existence of a torsion-free $G_2$-structure $\widetilde \varphi_t$ on $M$ such that $[\widetilde \varphi_t] = [\varphi_t] \in H^3(M)$ and $\| \widetilde \varphi_t - \varphi_t \|_{C^0} \leq \epsilon_1$, where $\epsilon_1 > 0$ is some fixed small constant. By taking $\epsilon_1$ small enough we can assume that $\|\widetilde \varphi - \varphi \|_{C^0} \leq \epsilon_1$ implies $2^{-1}g_{\varphi} \leq g_{\widetilde \varphi} \leq 2 g_{\varphi}$ for any $G_2$-structures $\widetilde \varphi, \varphi$. The following theorem gives sufficient conditions for the path $\{ \widetilde \varphi_t \}_{t \in (0,T]}$ to have finite energy and length in the moduli space:

\begin{thm}     \label{thm:finitelength}
    Let $\{\widetilde \varphi_t \}_{t \in (0,T]}$ be a continuous family of torsion-free $G_2$-structures and $\{\varphi_t\}_{t \in (0,T]}$ be a family of closed $G_2$-structures on $M$, such that $[\widetilde \varphi_t] = [\varphi_t] \in H^3(M)$ and $\| \widetilde \varphi_t - \varphi_t \|_{C^0} \leq \epsilon_1$ for all $t \in (0,T]$. We assume that $\varphi_0 = \left. \varphi_t \right|_{U_0}$ is independent of $t$, that each $Y_i$ is endowed with a family of closed $G_2$-structures $\{ \varphi_{i,t} \}_{t \in  (0,T]}$ such that $\left. \varphi_t \right|_{K_i} = \left.  \varphi_{i,t} \right|_{K_i}$ for all $t \in (0,T]$, and that the following assumptions are satisfied:
    \begin{enumerate}[(i)]
        \item $b^1(M) = 0$, and each boundary map $\delta_i : H_3(\overline Y_i, \del \overline Y_i) \rightarrow H_2(\del \overline Y_i)$ is trivial.
        \item For all $i$ and all $[C] \in H_3(Y_i)$, the function $f_{i,[C]}(t) = \int_{[C]} \varphi_{i,t}$ is of class $C^2$ and satisfies: 1) $f^\prime_{i,[C]} : (0,T] \rightarrow \R$ is uniformly bounded, and 2) $f^{\prime \prime}_{i,[C]} \in L^1((0,T])$.
        \item There exists a metric $g_i$ on each $Y_i$ such that $g_{\varphi_{i,t}} \leq g_i$ for all $t \in (0,T]$.
    \end{enumerate}
    Then $\{\widetilde \varphi_t\}_{t \in (0,T]}$ induces a path of class $C^2$ in $\M$ with finite energy and length. 
\end{thm}

\begin{proof}
    Let us consider the bases $[C_k],[C_{ij}] \in H_3(M)$ and $[D_k],[D_{ij}] \in H_4(M)$ provided by Lemma \ref{lem:support}, and denote by $[C_k^*],[C_{ij}^*] \in H^3(M)$ and $[D_k^*],[D_{ij}^*] \in H^4(M)$ their respective dual bases. The cohomology class $[\widetilde \varphi_t] \in H^3(M)$ reads
    \begin{equation*}
        [\widetilde \varphi_t] = [\varphi_t] = \sum_{k=1}^n f_k(t) [C_k^*] + \sum_{i=1}^m \sum_{j=1}^{n_i} f_{ij}(t) [C_{ij}^*]
    \end{equation*}
    where 
    \begin{equation*}
        f_k(t) = \int_{[C_k]} \varphi_t, ~~ \text{and} ~~ f_{ij}(t) = \int_{[C_{ij}]} \varphi_t.
    \end{equation*}
    Since the restriction $\left. \varphi_t \right|_{U_0}$ is constant and the homology classes $[C_1],\ldots,[C_n]$ are represented by cycles supported in $U_0$, the functions $f_1,\ldots,f_n$ are constant. Moreover, as the homology classes $[C_{ij}]$ are represented by cycles supported in $K_i$ and $\left. \varphi_t \right|_{K_i} = \left.  \varphi_{i,t} \right|_{K_i}$, we deduce that $f_{ij} = f_{i,[C_{ij}]}$, and thus $f_{ij}$ is of class $C^2$, $f^\prime_{ij}$ is uniformly bounded and $f^{\prime\prime}_{ij}$ is $L^1$. This implies in particular that $\{ \widetilde \varphi_t \}_{t \in (0,T]}$ induces a path of class $C^2$ in $\M$.

    Similarly, the cohomology class of the $4$-form can be written
    \begin{equation*}
        [\Theta(\widetilde \varphi_t)] = \sum_{k=1}^n g_k(t) [D_k^*] + \sum_{i=1}^m \sum_{j=1}^{n_i} g_{ij}(t) [D_{ij}^*]
    \end{equation*}
    where
    \begin{equation*}
        g_k(t) = \int_{[D_k]} \Theta(\widetilde \varphi_t) ~~ \text{and} ~~ g_{ij}(t) = \int_{[C_{ij}]} \Theta(\widetilde \varphi_t) .
    \end{equation*}
    As the bases $[C_k^*],[C^*_{ij}] \in H^3(M)$ and $[D^*_k],[D^*_{ij}] \in H^4(M)$ are dual for the cup-product and the functions $f_k$ are constant, it follows that:
    \begin{align*}
        \left \langle \frac{d [\widetilde \varphi_t]}{dt} \cup [\Theta( \widetilde \varphi_t)] , [M] \right \rangle & = \sum_{i=1}^m \sum_{j=1}^{n_i} g_{ij}(t) f^\prime_{ij}(t), ~ \text{and} \\
        \left \langle \frac{d^2 [\widetilde \varphi_t]}{dt^2} \cup [\Theta(\widetilde \varphi_t)] , [M] \right \rangle & = \sum_{i=1}^m \sum_{j=1}^{n_i} g_{ij}(t) f^{\prime\prime}_{ij}(t) .
    \end{align*}
    We remark that $\Vol(\varphi_t) \geq \int_{U_0} \varphi_0 > 0$ for all $t \in (0,T]$, and as $\| \widetilde \varphi_t - \varphi_t \|_{C^0} \leq \epsilon_1$ it follows that $\Vol(\widetilde \varphi_t)$ is uniformly bounded below away from zero. Moreover the functions $f^\prime_{ij}$ are uniformly bounded and the functions $f^{\prime\prime}_{ij}$ are $L^1$, and thus it is enough to show that the functions $g_{ij}$ are uniformly bounded to apply Corollary \ref{cor:incompbounds}. 
    
    Since $\widetilde \varphi_t$ is co-closed the $4$-form $\Theta(\widetilde \varphi_t)$ is a calibration, and by \eqref{eq:thetaest} we have
    \begin{equation*}
        |g_{ij}(t)| =\left| \int_{[D_{ij}]} \Theta(\widetilde \varphi_t) \right| \leq \Vol(D_{ij},g_{\widetilde \varphi_t}).
    \end{equation*}
    As $\|\widetilde \varphi_t - \varphi_t \|_{C^0} \leq \epsilon_1$ we have $g_{\widetilde \varphi_t} \leq 2 g_{\varphi_t}$, and thus $\left. g_{\widetilde \varphi_t} \right|_{K_i} \leq 2 \left. g_{\varphi_{i,t}} \right|_{K_i} \leq 2 \left. g_i \right|_{K_i}$. By Lemma \ref{lem:support}, we can assume that the $4$-cycle $D_{ij}$ is supported in $K_i$ and thus
    \begin{equation*}
        |g_{ij}(t)| \leq \Vol(D_{ij}, g_{\widetilde \varphi_t}) \leq 4 \Vol(D_{ij},g_i) .
    \end{equation*}
    Hence the functions $g_{ij}$ are uniformly bounded. Therefore the path $\{ \widetilde \varphi_t \}_{t \in (0,T]}$ satisfies the assumptions of Corollary \ref{cor:incompbounds} and the theorem follows.
\end{proof}

In the case of the generalised Kummer construction \cite{joyce1996compacti,joyce1996compactii}, $U$ is the complement of the singular set of a $G_2$-orbifold $T^7 / \Gamma$, where $\Gamma$ is a finite subgroup of $G_2$, and thus $U$ carries a flat $G_2$-structure $\varphi_0$. Each connected component of the singular set of $T^7/\Gamma$ is assumed to have a neighbourhood isometric to either one of
\begin{enumerate}
    \item $(T^3 \times B^4 / G_i) / F_i$, where $T^3$ is a flat $3$-torus, $B^4 \subset \C^2$ a Euclidean ball, $G_i$ a finite subgroup of $SU(2)$ acting freely on $\C^2$ except at the origin, and $F_i$ a group of isometries of $T^3 \times \C^2 / G_i$ acting freely on $T^3$; or
    \item $(S^1 \times B^6/G_i)/F_i$, where $S^1$ is a flat circle, $B^6 \subset \C^3$ a Euclidean ball, $G_i$ a finite subgroup of $SU(3)$ acting freely on $\C^3$ except at the origin, and $F_i$ a group of isometries of $S^1 \times \C^3 / G_i$ acting freely on $S^1$.
\end{enumerate}

The noncompact manifold $Y_i$ used to resolve a singularity of the first type is $Y_i = (T^3 \times X_i)/F_i$ where $X_i$ is an Asymptotically Locally Euclidean (ALE) space with holonomy $SU(2)$ asymptotic to $\C^2 / G_i$, equipped with an $F_i$-action such that $(T^3 \times X_i)/F_i$ is asymptotic to $(T^3 \times \C^2 / G_i) / F_i$. It has boundary $\Sigma_i = (T^3 \times S^3/G_i)/F_i$, and in particular $H_2(\Sigma_i) \simeq H_2(T^3)^{F_i}$. In addition, it follows from the K\"unneth theorem -- taking into account that $H_1(X_i) = 0$ since hyperk\"{a}hler ALE spaces are simply connected -- that $H_2(Y_i) \simeq H_2(T^3)^{F_i} \oplus H_2(X_i)^{F_i}$. Thus the map $H_2(\del \overline Y_i) \rightarrow H_2(\overline Y_i)$ is injective, which implies that the boundary map $\delta_i : H_3(\overline Y_i, \del \overline Y_i) \rightarrow H_2(\del \overline Y_i)$ is trivial. The manifold $Y_i$ is endowed with a family of torsion-free $G_2$-structures lifting to $T^3 \times X_i$ as
\begin{equation*}
    \varphi_{i,t} = \theta_1 \wedge \theta_2 \wedge \theta_3 + t^2 (\theta_1 \wedge \omega_{i,1} + \theta_2 \wedge \omega_{i,2} + \theta_3 \wedge \omega_{i,3})
\end{equation*}
where $(\theta_1,\theta_2,\theta_3)$ is a basis of harmonic $1$-forms on $T^3$ and $(\omega_{i,1},\omega_{i,2},\omega_{i,3})$ is an ALE hyperk\"{a}hler triple on $X_i$. It follows that for any homology class $[C] \in H_3(Y_i)$ we have $f_{i,[C]}(t) = a_{i,[C]} + b_{i,[C]}t^2$ for some constants $a_{i,[C]}, b_{i,[C]} \in \R$. Moreover, the associated metric on $T^3 \times X_i$ is a product $g_{T^3} + t^2 g_{X_i}$, where $g_{T^3}$ is a flat metric on $T^3$ and $g_{X_i}$ is an ALE metric on $X_i$. In particular $g_{\varphi_{i,t}} \leq g_{\varphi_{i,1}}$ for all $t \in (0,1]$. 

For the second type of singularities, the manifold $Y_i$ is of the form $Y_i = (S^1 \times Z_i) / F_i$ where $Z_i$ is an ALE manifold with holonomy $SU(3)$ asymptotic to $\C^3 / G_i$, equipped with an $F_i$-action such that $(S^1 \times Z_i)/F_i$ is asymptotic to $(S^1 \times \C^3/G_i)/F_i$. Here the boundary of $\overline Y_i$ is $\Sigma_i = (S^1 \times S^5/G_i)/F_i$ so that $H_2(\Sigma_i) = 0$, which in particular implies that the boundary map $\delta_i : H_3(\overline Y_i, \del \overline Y_i) \rightarrow H_2(\del \overline Y_i)$ is trivial. There is a family of torsion-free $G_2$-structures on $Y_i$ which lifts to $S^1 \times Z_i$ as
\begin{equation*}
    \varphi_{i,t} = t^2 \theta \wedge \omega_i + t^3 \re \Omega_i
\end{equation*}
on $Y_i$, where $\theta$ is a nontrivial harmonic form on $S^1$, $\omega_i$ a K\"{a}hler form and $\Omega_i$ a holomorphic volume form on $Z_i$ such that $(\omega_i,\Omega_i)$ is an ALE torsion-free $SU(3)$-structure. Thus for any homology class $[C] \in H_3(Y_i)$ there are constants $a_{i,[C]}$ and $b_{i,[C]}$ such that $f_{i,[C]}(t) = a_{i,[C]} t^2 + b_{i,[C]} t^3$. As in the previous case, the metric $g_{\varphi_{i,t}}$ lifts to the product $g_{S^1} + t^2 g_{Z_i}$ on $S^1 \times Z_i$ and hence $g_{\varphi_{i,t}} \leq g_{\varphi_{i,1}}$ for $t \in (0,1]$.

When the gluing data of the generalised Kummer construction is chosen so that $b^1(M) = 0$, all of the assumptions of our theorem are satisfied and thus the degeneration to $T^7/\Gamma$ corresponds to a finite-distance limit in the moduli space.

\begin{cor}     \label{cor:incomplete}
    The generalised Kummer $G_2$-manifolds constructed in \cite{joyce1996compacti,joyce1996compactii} have incomplete moduli spaces.
\end{cor}

    \section{The volume of cycles and infinite-distance limits}     \label{section:infinite}

An interesting special case in the generalised Kummer construction is when all the singularities of $T^7/\Gamma$ are resolved by gluing quotients of products of a $3$-torus and a hyperk\"{a}hler ALE space. In that case,  the cohomology class $[\varphi_t]$ is an affine function of $t^2$. This leads us to consider the general situation of a path of torsion-free $G_2$-structures $\{\varphi_t\}_{t \in (0,T]}$ whose cohomology classes form a line segment in $H^3(M)$.

To further motivate this question, let us compare with the case of K\"{a}hler cones, which we mentioned in introduction. Let $X$ be a compact K\"{a}hler manifold of (complex) dimension $n \geq 2$ and let $\K$ be the cone of K\"{a}hler classes on $X$. This is an open convex cone of the space of real $(1,1)$-forms. Let us denote by $\overline \K$ the closure of $\K$ in $H^{1,1}(X;\R)$, and $\del \K = \overline \K \backslash \K$. The K\"{a}hler cone has a natural metric, which comes from the Hessian potential $- \log \int_X \omega^n / n! = - \log \Vol(\omega)$. A natural question to ask is which classes $\alpha \in \del \K$ correspond to infinite-distance limits for this metric. The answer is simple: either $\int_X \alpha^n = 0$, in which case this is an infinite-distance limit; or $\int_X \alpha^n > 0$ and the limit is at finite distance \cite{magnusson2012metriques}. Let us give a brief sketch of the argument. If $\int_X \alpha^n = 0$, then any path $\omega_t \in \K$ converging to $\alpha \in H^{1,1}(X;\R)$ has $\int_X \omega_t^n \rightarrow 0$, and such limit is at infinite distance. Otherwise $\int_X \alpha^n > 0$, and if $\omega \in \K$ we can consider the path $\omega_t = \alpha + t\omega$. For $t \in (0,\infty)$ we have $\omega_t \in \K$, and moreover $\frac{d \omega_t}{dt} = \omega$ and $\frac{d^2 \omega_t}{dt^2} = 0$. Arguing along the same lines as in the proof of Proposition \ref{prop:energy}, one can see that for any $\tau \in (0,1]$ the energy of the path $\{ \omega_t \}_{t \in (\tau,1]}$ is
\begin{equation*}
    E_\tau^1(\omega_t) = \frac{\int_X \omega \wedge (\alpha + \tau \omega)^{n-1}}{(n-1)! \Vol(\omega_\tau)} - \frac{\int_X \omega \wedge (\alpha + \omega)^{n-1}}{(n-1)! \Vol(\omega_1)} \cdot
\end{equation*}
Since $\int_X \alpha^n > 0$, $\Vol(\omega_\tau)$ is uniformly bounded below away from zero as $\tau \rightarrow 0$, and as the numerator of the first term is a polynomial function of $\tau$ it remains bounded as $\tau \rightarrow 0$. It follows that the path $\{ \omega_t \}_{t \in (0,1]}$ has finite energy and length in $\K$, and thus $\alpha$ is a finite-distance limit.

Let us now go back to the $G_2$-case, and consider family of torsion-free $G_2$-structures $\{\varphi_t\}_{t \in (0,T]}$ on a compact $7$-manifold $M$ with $b^1(M) = 0$, inducing a smooth path in the moduli space $\M$. We assume that $[\varphi_t] = [\varphi_0] + t [\dot \varphi ]$, where $[\varphi_0],[\dot \varphi] \in H^3(M)$ are fixed cohomology classes. Heuristically, we want to think of $[\varphi_0]$ as lying on the boundary of the image of $\M$ in $H^3(M)$. As $\ddot \varphi_t = 0$ we have $\frac{d^2 \F(\varphi_t)}{dt^2} = D^2_{\varphi_t} \F(\dot \varphi_t, \dot \varphi_t) \geq 0$, and thus the function $\F(\varphi_t) = - 3 \log \Vol(\varphi_t)$ is convex on $(0,T]$. Hence $\Vol(\varphi_t)$ is bounded above, and either $\Vol(\varphi_t) \rightarrow 0$ as $t \rightarrow 0$ or it is uniformly bounded below away from zero. If the volume shrinks to zero, then the limit of $\varphi_t$ as $t \rightarrow 0$ is at infinite distance in the moduli space. The remaining interesting case is when $\Vol(\varphi_t)$ is uniformly bounded below away from zero. As $\frac{d^2[\varphi_t]}{dt^2} = 0$, the energy takes a particularly simple form:
\begin{equation}    \label{eq:energyaff}
    E_\tau^T(\varphi_t) = \frac{\langle [\dot \varphi] \cup [\Theta(\varphi_\tau)] , [M] \rangle}{\Vol(\varphi_\tau)} - \frac{\langle [\dot \varphi] \cup [\Theta(\varphi_T)] , [M] \rangle}{\Vol(\varphi_T)} \cdot
\end{equation}
So far the situation is very similar to that of K\"{a}hler cones, except for one crucial difference, which is that $[\Theta(\varphi_\tau)]$ has no reason to be a merely polynomial function of $\tau$. Thus we do not know whether the numerator of the first term remains bounded as $\tau \rightarrow 0$, and it may be that the limit of $\varphi_t$ as $t \rightarrow 0$ is at infinite distance in the moduli space even though the volume is bounded below. Such a phenomenon, if it occurs, would be a feature of $G_2$-moduli spaces that has no analogy in the geometry of K\"{a}hler cones.

To gain more insight into the geometry of such a situation, let us denote by $\PD[\dot \varphi] \in H_4(M)$ the Poincar\'{e}-dual class of $[\dot \varphi]$. Then we can bound the numerators in \eqref{eq:energyaff} by
\begin{equation}
    |\langle [\dot \varphi] \cup [\Theta(\varphi_t)] , [M] \rangle | = \left| \int_{\PD [\dot \varphi]} \Theta(\varphi_t) \right| \leq \Vol(\PD[\dot \varphi], g_{\varphi_t})
\end{equation}
where we define $\Vol(\PD[\dot \varphi], g_{\varphi_t})$ as the infimum of $\Vol(D,g_{\varphi_t})$ taken over the $4$-cycles $D$ representing $\PD[\dot \varphi] \in H_4(M)$. If $\Vol(\PD[\dot \varphi],g_{\varphi_t})$ is bounded, then we easily deduce that the energy of $\{\varphi_t\}_{t \in (0,T]}$ is finite. As the energy $E_{\tau}^T(\varphi_t)$ is a decreasing function of $\tau$, it is in fact enough to assume that $\Vol(\PD[\dot \varphi],g_{\varphi_{t_i}})$ is uniformly bounded for some sequence $t_i \rightarrow 0$. By contrapositive, we obtain the following condition:

\begin{prop}    \label{prop:infinite}
    Let $\{\varphi_t\}_{t \in (0,T]}$ be a family of torsion-free $G_2$-structures on $M$ inducing a smooth path in $\M$, and suppose that the cohomology class $\frac{d [\varphi_t]}{dt} = [\dot \varphi]$ is constant in $H^3(M)$ and that the volume of $(M,\varphi_t)$ is uniformly bounded below away from zero. If the limit of $\varphi_t$ as $t \rightarrow 0$ is at infinite distance in the moduli space, then
    \begin{equation*}
        \Vol(\PD[\dot \varphi], g_{\varphi_t}) \rightarrow \infty ~~ \text{as} ~~ t \rightarrow 0.
    \end{equation*}
\end{prop}

Heuristically, this result says that if there is a point of the boundary of $\M$ that can be approached by a path of torsion-free $G_2$-structures whose cohomology classes form a line segment in $H^3(M)$, then there is the following trichotomy. Either the volume is shrinking to zero along the path, in which case the limit is at infinite distance; or the volume is bounded below away from zero and the length of the path is infinite, and in that case there must be a homology class in degree $4$ whose volume is going to infinity; or this is a finite-distance limit. Some of the generalised Kummer $G_2$-manifolds provide examples of the third case, and the first case occurs for instance by scaling any torsion-free $G_2$-structure. However, we do not know if the second case can happen, and as previously noted this would have to be a phenomenon specific to $G_2$-moduli spaces, by contrast with what can occur at the boundary of K\"{a}hler cones.

Proposition \ref{prop:infinite} can be generalised to families of torsion-free $G_2$-structures whose cohomology classes define a path in $H^3(M)$ which is regular enough. For instance, if we assume that the path of cohomology classes is $C^2$, has bounded first derivative and integrable second derivative, and that the volume of $M$ is bounded below away from zero, it is straightforward to deduce from Corollary \ref{cor:incompbounds} and the bound \eqref{eq:thetaest} that if all homology classes in degree $4$ have bounded volume then the energy of the path is finite. This could also be used to deduce the incompleteness of the moduli spaces of the generalised Kummer $G_2$-manifolds, with a criterion that is easier to state than the assumptions of Theorem \ref{thm:finitelength}, but which actually takes more effort to check in practice since one would have to consider all homology classes of degree $4$. This also indicates that $G_2$-manifolds constructed by resolution of isolated conical singularities \cite{karigiannis2009desingularization} or the Joyce--Karigiannis construction \cite{joyce2021new} would be expected to have incomplete moduli spaces. In these cases it is clear that the total volume is bounded below away from zero and that all homology classes have bounded volume, and the only assumption that would be left to check is the one concerning the derivatives of the path of cohomology classes.

We finish with a few remarks and open questions. Since the moduli space is generally not complete, we can define its metric completion $\overline{\M}$ and the (finite-distance) boundary $\del \M = \overline \M \backslash \M$ may be nonempty. It would be interesting to know if the points of the boundary have a geometric interpretation. For the generalised Kummer $G_2$-manifolds, we considered paths of torsion-free $G_2$-structures such that the associated metrics are converging to a flat $G_2$-orbifold in the Gromov--Hausdorff topology. Hence a first question to ask is whether a path of torsion-free $G_2$-structures which has finite length in the moduli space always converges to a compact and possibly singular $G_2$-manifold in the Gromov--Hausdorff sense. Another related problem would be to give sufficient conditions for the limit of a path (or a sequence) in the moduli space to be at infinite distance. Since this is always the case when the volume diverges to $0$ or $\infty$, we need to fix the volume for this question to be interesting. In the case of twisted connected sums for instance, the volume grows linearly with the length of the neck region and therefore this represents an infinite-distance limit in the moduli space -- even though we could deduce from Proposition \ref{prop:energy} that the energy of the obvious path is bounded, which would be consistent with the speed of the path being proportional to the inverse of the neck length. However, if we normalise the volume it is no longer clear that the distance along the path is unbounded. The main challenge is that in order to find a lower bound on the distance between two points in the moduli space, we need to control the length of \emph{all} paths connecting them, whereas the length of one particular path is enough to give an upper bound. The author hopes to come back to these questions in future research.

    \section*{Acknowledgements}

This work is supported by scholarships from the Clarendon Fund and the Saven European Programme.

\small

\bibliographystyle{abbrv}
\bibliography{biblio_incompleteness}

\end{document}